\documentclass[aos,preprint]{imsart}
\setattribute{journal}{name}{}

\RequirePackage[T1]{fontenc}
\RequirePackage{amsthm,amsmath,amssymb,bbm,mathrsfs,dsfont,wasysym}
\RequirePackage[authoryear]{natbib}
\RequirePackage[colorlinks,citecolor=blue,urlcolor=blue]{hyperref}
\usepackage[french,english]{babel} % sans option, babel choisit la langue en function of celle définie in la classe du document
\usepackage[latin1]{inputenc} 
\usepackage[babel=true]{csquotes} % csquotes va utiliser la langue définie in babel
 
\startlocaldefs
\numberwithin{equation}{section}
\theoremstyle{plain}

\usepackage{needspace}

\theoremstyle{plain}
\newtheorem{lemme}{Lemma}[section]
\newtheorem{proposition}{Proposition}[section]
\newtheorem{theoreme}{Theorem}[section]

\newtheorem{hypoth}{Assumption}[section]

\newenvironment{hypo}{% %Permet d'éviter a saut of page après l'entête
    \Needspace*{3\baselineskip}%
    \hypoth
}{\endhypoth}

\theoremstyle{remark}
\newtheorem{remarque}{Remark}[section]

\newtheorem{definition}{Definition}[section]

\endlocaldefs

\newcommand{\R}{\mathbbm{R}}
\newcommand{\N}{\mathbbm{N}}
\newcommand{\J}{\mathrm{J}}

\newcommand{\E}{\mathrm{E}}
\renewcommand{\P}{\mathrm{P}}
\renewcommand{\L}{\mathrm{L}}
\newcommand{\M}{\mathrm{M}}

\newcommand{\trace}{\mathrm{trace}}

\begin{document}

\begin{frontmatter}
\title{Efficient prediction in $\L^2$-differentiable families of distributions}
\runtitle{Efficient prediction in $\L^2$-differentiable families}

\begin{aug}
\author{\fnms{Emmanuel} \snm{Onzon}\ead[label=e1]{emmanuel.onzon@univ-lyon1.fr}},
\runauthor{E. Onzon}

\affiliation{Université Lyon 1}

\address{Bâtiment Polytech Lyon,\\ 15 Boulevard André Latarget,\\
 69100 Villeurbanne, France\\
\printead{e1}\\
\phantom{E-mail:\ }}
\end{aug}

\begin{abstract}\quad
A proof of the Cramér-Rao inequality for prediction is presented under conditions of $\L^2$-differentiability of the family of distributions of the model. The assumptions and the proof differ from those of \cite{miyata2001} who also proved this inequality under $\L^2$-differentiability conditions. It is also proved that an efficient predictor (\emph{i.e.} which risk attains the bound) exists if and only if the family of distributions is of a special form which can be seen as an extension of the notion of exponential family. This result is also proved under $\L^2$-differentiability conditions.
\end{abstract}



\end{frontmatter}

\section{Introduction}

%\subsection{Background}

Statistical prediction relates to the inference of an unobserved random quantity from observations, it is considered here as an extension of point estimation, where the quantity to infer is not necessarily deterministic.
We follow the framework posed by \cite{yatracos1992}. In full generality, the problem of statistical prediction is to estimate a quantity $g(X,Y,\theta)$, we shall say \emph{predict} $g(X,Y,\theta)$, where $X$ is an observed random variable representing the observations, $Y$ an unobserved random variable and $\theta$ the parameter of the model $\{ \P_\theta \, | \, \theta \in \Theta \}$ which the distribution of $(X,Y)$ is supposed to belong to. We shall assume that $g$ takes its values in $\R^k$ and $\Theta \subset \R^d$.
That framework encompasses a wide variety of statistical problems ranging from stochastic processes prediction and time series forecasting (\cite{Johansson1990}, \cite{adke1997}, \cite{bosq2012}, \cite{onzon2013a}) to latent variable models and random effects inference (\cite{nayak2000}, \cite{nayak2003}).
If $p(X)$ is used to predict $g(X,Y,\theta)$ we shall call it a predictor and measure its performance with its mean squared error of prediction which breaks down in the following sum
\[
\E_\theta (p(X) - g(X,Y,\theta))^{\times 2} = \E_\theta (p(X) - r(X,\theta))^{\times 2} + \E_\theta (r(X,\theta) - g(X,Y,\theta))^{\times 2},
\]
with $r(X,\theta) = E_\theta[g(X,Y,\theta) | X]$ and where we use the notation $A^{\times 2} = AA'$ the product of a matrix with its transpose. The second term of the right hand side is incompressible, it does not depend on the choice of the predictor. Hence from now on we are interested in the first term which we call quadratic error of prediction (QEP). More generally, we shall investigate the problem of predicting a quantity $g(X,\theta)$ without refering it is a conditional expectation and call QEP the quantity
\[
R(\theta) = \E_\theta (p(X) - g(X,\theta))^{\times 2}.
\]

A lower bound of Cramér-Rao type has been proved for the QEP with conditions of point differentiability of the family of the densities of the distributions of the model with respect to the parameter and conditions of differentiability under the integral sign (\cite{yatracos1992}, \cite{nayak2002}, \cite{bosq-blanke2007}). The bound has also been proved for conditions of $\L^2$-differentiability of the family of distributions of the model (\cite{miyata2001}, \cite{onzon2012}). In the one-dimensional case ($k=d=1$) and for unbiased predictors it reads
\[
\E_\theta (p(X) - g(X,\theta))^2 \geqslant \frac{ \big( \E_\theta \partial_\theta g(X,\theta) \big)^2 }{ I(\theta) },
\]
where $I(\theta)$ is the Fisher information. We prove this inequality under conditions of $\L^2$-differentiability of the family of distributions of the model in Section~\ref{section_CR_ineg}. The set of assumptions we use here is different from those made by \cite{miyata2001}, for instance there is no reference to the random variable $Y$ in our assumptions while \cite{miyata2001} uses the distribution of the couple $(X,Y)$.

When the mean squared error of an estimator attains the Cramér-Rao bound we say that it is \emph{efficient}. By analogy, an efficient predictor is a predictor which QEP attains the Cramér-Rao bound. In the case of estimation it is proved that there exists an efficient estimator $\delta(X)$ of $\psi(\theta) \in \R^d$ if and only if the family of distributions of the model is exponential, \emph{i.e.} of the form
\[
\frac{\mathrm{d}\P_\theta}{\mathrm{d}\P_{\theta_0}}(x) = \exp \{ A(\theta)'\delta(x) - B(\theta) \},
\]
for some $\theta_0 \in \Theta$, and differentiable functions $A:\Theta\to \R^k$ and $B:\Theta \to \R$, with $(\J_\theta A(\theta))' = I(\theta) (\J_\theta \psi(\theta))^{-1}$ and $\nabla_\theta B(\theta) = (\J_\theta A(\theta))' \psi(\theta)$. The result has been proved under different conditions (\cite{wijsman1973}, \cite{fabian1977}, \cite{muller-funk1989}).

An analogous result for prediction appears in \cite{bosq-blanke2007} in the one-dimensional case and in \cite{onzon2011} in the multidimensional case. In both cases the result is proved under conditions of point differentiability of the family of the densities of the distributions of the model and differentiability under the integral sign. For this result the family is not necessarily exponential but has a form which may be seen as an extension of the notion of exponential family. There exists an efficient predictor $p(X)$ to predict $g(X,\theta) \in \R^d$ if and only if
\[
\frac{\mathrm{d}\P_\theta}{\mathrm{d}\P_{\theta_0}}(x) = \exp \{ A(\theta)'p(x) - B(x,\theta) \},
\]
for some $\theta_0 \in \Theta$, and differentiable functions $A:\Theta\to \R^k$ and $B:\Theta \times E \to \R$, with $(\J_\theta A(\theta))' = I(\theta) (\E_\theta \J_\theta g(X,\theta))^{-1}$ and $\nabla_\theta B(x,\theta) = (\J_\theta A(\theta))' g(x,\theta)$.
Section~\ref{section_pred_eff} presents a proof of this result under $\L^2$-differentiability conditions. The proof is based on the proof of the result for estimation that appears in \cite{muller-funk1989}.

The Appendices gather definitions and results on $\L^2$-differentiability and uniform integrability that are used throughout the paper.

\section{The Cramér-Rao inequality for prediction in $\L^2$-differentiable families}\label{section_CR_ineg}

The following lemma gives a matrix inequality on which the proof of the Cramér-Rao inequality is based.

\begin{lemme}\label{lemme_inegalite_covariance}
Let $T$ and $S$ be random variables taking values in $\R^k$ and $\R^d$ respectively, such that $\E \|T\|_{\R^k}^2 < \infty$ and $\E \|S\|_{\R^d}^2 < \infty$, and such that $\E S^{\times 2}$ is an invertible matrix. Then the following inequality holds,
\begin{equation}\label{equation_ineg_cov}
\E T^{\times 2} \geqslant \E (TS') (\E S^{\times 2})^{-1} \E (ST').
\end{equation}
The equality holds in (\ref{equation_ineg_cov}) iff
\begin{equation}\label{equation_egalite_covariance}
T = \E (TS')(\E S^{\times 2})^{-1}S, \quad \text{a.s.}
\end{equation}
\end{lemme}

\begin{proof}
Let $Z$ be the random vector taking values in $\R^k$ defined as follows
\[
Z = T - \E (TS') (\E S^{\times 2})^{-1}S.
\]
Then its matrix of moment of order $2$ is
\[
\E Z^{\times 2} = \E T^{\times 2} - \E (TS') (\E S^{\times 2})^{-1} \E (ST').
\]
Let $x\in\R^k$, then
\[
x'(\E Z^{\times 2})x = \E (x'Z Z'x) = \E (Z'x)'(Z'x) = \E \| Z'x \|_{\R^d}^2 \geqslant 0.
\]
Hence for all $x\in\R^k$,
\[
x' (\E T^{\times 2} - \E (TS') (\E S^{\times 2})^{-1} \E (ST')) x \geqslant 0.
\]
We deduce (\ref{equation_ineg_cov}).

Suppose the equality holds in (\ref{equation_ineg_cov}). Then $\E Z^{\times 2}=0$, hence $\trace(\E Z^{\times 2})=0$, hence $\E(\trace(Z^{\times 2}))=0$. Yet
\[
\trace(Z^{\times 2})=\trace(Z' Z) = Z' Z = \|Z\|_{\R^k}^2.
\]
Hence $\E \|Z\|_{\R^k}^2 = 0$. Hence $Z = 0$ almost surely. We deduce (\ref{equation_egalite_covariance}).

Suppose (\ref{equation_egalite_covariance}) holds. Then $Z = 0$ almost surely. Hence $\E Z^{\times 2}=0$, the equality in (\ref{equation_ineg_cov}) ensues.
\end{proof}

\begin{remarque}\emph{Geometric interpretation of the matrix inequality}\\
The inequality of Lemma~\ref{lemme_inegalite_covariance} may be interpreted as a Bessel type inequality in the space of random variables with finite moment of order 2. More precisely, consider
\[
\L^2_\P = \{ U \text{ real r.v. } | \, \E U^2 < \infty \},
\]
and the following endomorphism of $\L^2_\P$
\[
P_S : U \mapsto \E (US') (\E S^{\times 2})^{-1} S,
\]
Then one may show that $P_S$ is the orthogonal projection on the space generated by the components of $S$. Indeed, it satisfies $P_S \circ P_S = P_S$, and any component of $S$ is stable by $P_S$, and $P_S$ is self-adjoint, for all $U, V \in L^2_\P$,
\[
\E \left( P_S(U) V \right) = \E \left( U P_S(V) \right).
\]
Then Pythagoras' theorem implies then that for all $U \in \L^2_\P$,
\[
\E U^2 \geqslant \E \left( P_S(U) \right)^2.
\]
We deduce that for all $x\in\R^k$
\[
x' \E T^{\times 2} x = \E (x'T)^2 \geqslant \E \left( P_S(x'T) \right)^2
= x' \E (TS') (\E S^{\times 2})^{-1} \E (ST') x,
\]
with $T$ defined as in Lemma~\ref{lemme_inegalite_covariance}. We deduce the inequality~(\ref{equation_ineg_cov}).

There is equality in (\ref{equation_ineg_cov}) iff for all $x\in\R^k$, $x'T$ is invariant by $P_S$, \emph{i.e.},
\[
\E \left(x'TS'\right) (\E S^{\times 2})^{-1} S = x'T.
\]
We deduce (\ref{equation_egalite_covariance}).
\end{remarque}

\begin{lemme}\label{lem_ineg_CR_L2_diff}
Let $(\mathcal{X},\mathcal{B},\P_\theta, \theta\in\Theta)$ be a model, $\theta_0\in\mathring{\Theta}$, $p(X)$ a predictor of $g(X,\theta)$ taking values in $\R^k$, and $U(\theta_0)$, a neighbourhood of $\theta_0$, which fulfills the following conditions.
\begin{enumerate}
\item The family $(\P_\theta,\theta\in\Theta)$ is $\L^2$-differentiable at $\theta_0$, with derivative $\dot{L}_{\theta_0}$.
\item Fisher information matrix $I(\theta_0)$ is invertible.
\item $\sup_{\theta\in U(\theta_0)} \E_\theta \| p(X) \|_{\R^k}^2 < \infty$.
\end{enumerate}
Then the fonction $\psi : \theta \mapsto \E_\theta p(X)$ is differentiable at $\theta_0$, and the QEP of $p(X)$ at $\theta_0$ satisfies the following inequality.
\begin{equation}\label{equation_ineg_CR_pred_L2_diff}
\E_{\theta_0} (p(X) - g(X,\theta_0))^{\times 2} \geqslant
G(\theta_0) I(\theta_0)^{-1} G(\theta_0)',
\end{equation}
with 
\begin{equation}\label{equation_G_CR_ineg_lemma}
G(\theta_0) = \J_\theta \psi(\theta_0) - \E_\theta g(X,\theta_0) \dot{L}'_{\theta_0},
\end{equation}
The equality holds in (\ref{equation_ineg_CR_pred_L2_diff}) iff
\begin{equation*}%\label{equation_egalite_CR}
p(X) = g(X,\theta_0) + G(\theta_0) I(\theta_0)^{-1} \dot{L}'_{\theta_0},
\quad \P_{\theta_0}\text{-a.s.}
\end{equation*}
\end{lemme}

The symbol $\J_\theta$ denotes the jacobian matrix operator.

\begin{proof}
We set $S = \dot{L}_{\theta_0}$ and $T = p(X) - g(X,\theta_0)$. Using Proposition~\ref{prop_1_111_Liese} (applied to $\delta = p(X)$) one obtains
\[
\E_{\theta_0} (TS') = \E_{\theta_0}  (p(X) - g(X,\theta_0)) \dot{L}'_{\theta_0}
= \J_\theta \psi(\theta_0) - \E_{\theta_0} g(X,\theta_0) \dot{L}'_{\theta_0}.
\]
Then the result follows from Lemma~\ref{lemme_inegalite_covariance}.
\end{proof}

Lemma~\ref{lem_ineg_CR_L2_diff} gives a matrix inequality of Cramér-Rao type for predictors. However under some conditions, the matrix $G(\theta_0)$ that appears in the right hand side of the inequality has the following simpler form 
\begin{equation}%\label{equation_G_CR_ineg}
G(\theta_0) = \E_\theta \J_\theta g(X,\theta_0),
\end{equation}
instead of the form (\ref{equation_G_CR_ineg_lemma}).
We now proceed to obtain those conditions in the setup of $\L^2$-differentiable families of distributions.

\begin{proposition}\label{prop_deriv_g_L2_diff}
Let $(\mathcal{X},\mathcal{B},\P_\theta, \theta\in\Theta)$ be a model and  $\theta_0\in\mathring{\Theta}$, such that $(\P_\theta, \theta\in\Theta)$ is $\L^2$-differentiable at $\theta_0$. Let $g : \mathcal{X} \times \Theta \to \R^k$ be a function such that for all $\theta\in\Theta$, $g(\cdot,\theta)$ is measurable. Suppose there is $U(\theta_0)$, a neighbourhood of $\theta_0$, such that the following conditions hold.
\begin{enumerate}
\item For all $\theta$, $\theta' \in U(\theta_0)$, $g(X,\cdot)$ is $\P_\theta$-almost surely differentiable at $\theta'$ and \[
\sup_{(\theta,\theta') \in U(\theta_0)^2} \E_\theta \| \J_\theta g(X,\theta') \|_{\M_{k,d}}^2 < \infty.
\]
\item $\sup_{(\theta,\theta') \in U(\theta_0)^2} \E_\theta L_{\theta, \theta'}^2 < \infty$
\end{enumerate}
Then
\[
\J_\theta \E_{\theta_0} g(X,\theta_0) = \E_{\theta_0} \J_\theta g(X,\theta_0) + \E_{\theta_0} g(X,\theta_0) \dot{L}'_{\theta_0}.
\]
\end{proposition}

\begin{proof}
We first assume that $k=1$, \emph{i.e.} that $g:\mathcal{X}\times\Theta\to\mathbb{R}$.
We prove that
\[
\nabla_\theta \E_{\theta_0} g(X,\theta_0) = \E_{\theta_0} g(X,\theta_0) \dot{L}_{\theta_0} + \E_{\theta_0} \nabla_\theta g(X,\theta_0).
\]
%Nous adaptons la preuve of la Proposition 1.111 of \cite{liese2008}.

Let $a\in\R^d$, such that $\| a \| = 1$ and $(u_n, n\in\N)$ a sequence in $\R^d$ such that
\[
u_n \underset{n\to\infty}{\longrightarrow} 0,
\quad\text{and}\quad
\frac{u_n}{\|u_n\|} \underset{n\rightarrow \infty}{\longrightarrow} a.
\]
We set
\begin{align*}
\Delta_n & = \|u_n\|^{-1} \Big( \E_{\theta_0+u_n} g(X,\theta_0+u_n) - \E_{\theta_0}g(X,\theta_0) \Big)\\
& \quad
-a' \, \E_{\theta_0} \Big( g(X,\theta_0)\dot{L}_{\theta_0} + \nabla_\theta g(X,\theta_0) \Big).
\end{align*}
We prove that $\Delta_n \underset{n\rightarrow \infty}{\longrightarrow} 0$.
\begin{align*}
\Delta_n
& = \|u_n\|^{-1} \Big( \E_{\theta_0+u_n} g(X,\theta_0) - \E_{\theta_0} g(X,\theta_0) \Big) - \E_{\theta_0} \Big( a' \dot{L}_{\theta_0} g(X,\theta_0) \Big) \\
& \quad + \E_{\theta_0} \Big( \|u_n\|^{-1} L_{\theta_0}(u_n) \big( g(X,\theta_0 + u_n) - g(X,\theta_0) \big) \, - \, a'\nabla_\theta g(X,\theta_0)  \Big).
\end{align*}
From Proposition~\ref{prop_1_111_Liese} the following convergence holds (taking $\delta=g(X,\theta_0)$).
\[
\|u_n\|^{-1} \Big( \E_{\theta_0+u_n} g(X,\theta_0) - \E_{\theta_0} g(X,\theta_0) \Big) - \E_{\theta_0} \Big( a' \dot{L}_{\theta_0} g(X,\theta_0) \Big)
\underset{n\rightarrow \infty}{\longrightarrow} 0
\]
We set
\[
\tilde{\Delta}_n
= \|u_n\|^{-1} L_{\theta_0}(u_n) \big( g(X,\theta_0 + u_n) - g(X,\theta_0) \big) \, - \, a'\nabla_\theta g(X,\theta_0).
\]
To complete the proof it remains to prove that $\E_{\theta_0} \tilde{\Delta}_n \underset{n\rightarrow \infty}{\longrightarrow} 0$.
\begin{align*}
\tilde{\Delta}_n & = L_{\theta_0}(u_n) \Big( \|u_n\|^{-1} \big( g(X,\theta_0+u_n) - g(X,\theta_0) \big) \, - \, a'\nabla_\theta g(X,\theta_0) \Big)\\
& \quad + a'\nabla_\theta g(X,\theta_0) \big(L_{\theta_0}(u_n)-1\big).
\end{align*}

Let $U_n = a'\nabla_\theta g(X,\theta_0) \big( L_{\theta_0}(u_n) - 1 \big)$ then,
\[
\E_{\theta_0} |U_n| \leqslant
\E_{\theta_0} | a'\nabla_\theta g(X,\theta_0) | + \E_{\theta_0+u_n} | a' \nabla_\theta g(X,\theta_0) |.
\]
Yet for all $n$ large enough, $\theta_0+u_n\in U(\theta_0)$. Hence $(U_n)_{n\in\mathbb{N}}$ is bounded in $\L^1$. We prove that $(U_n, n\in\N)$ is uniformly integrable. Let $A\in\mathcal{B}$ then
\begin{align*}
\sup_{n\in\N} \E_{\theta_0} | \mathbbm{1}_A U_n | & \leqslant
 \, \P_{\theta_0}(A) \, \big( \E_{\theta_0} \| \nabla_\theta g(X,\theta_0) \|^2 \big)^{1/2}\\
& \quad +  \, \P_{\theta_0 + u_n}(A) \, \big( \E_{\theta_0+u_n} \|\nabla_\theta g(X,\theta_0) \|^2 \big)^{1/2},
\end{align*}
and
\[
\P_{\theta_0+u_n}(A) \leqslant \P_{\theta_0}(A) \, \big( \E_{\theta_0}(L_{\theta_0}(u_n) )^2 \big)^{1/2}.
\]
Hence the sequence $(U_n)_{n\in\mathbb{N}}$ is equicontinuous. Hence it is uniformly integrable. Yet
\begin{equation}\label{equa_L_conv_1}
L_{\theta_0}(u_n) \xrightarrow[n\to\infty]{\P_{\theta_0}} 1,
\end{equation}
from Lemma~\ref{lem_1_106_Liese}. Hence
$U_n \xrightarrow[n\to\infty]{\P_{\theta_0}} 0$. We deduce
\[
\E_{\theta_0} a'\nabla_\theta g(X,\theta_0) \big( L_{\theta_0}(u_n) - 1 \big)
\underset{n\rightarrow \infty}{\longrightarrow} 0.
\]

Moreover
\[
\|u_n\|^{-1} \big( g(X,\theta_0 + u_n) - g(X,\theta_0) \big) - a' \nabla_\theta g(X,\theta_0)
\xrightarrow[n\to\infty]{\P_{\theta_0}} 0.
\]
Combining with (\ref{equa_L_conv_1}) one obtains
\[
L_{\theta_0}(u_n) \Big( \|u_n\|^{-1} \big( g(X,\theta_0 + u_n) - g(X,\theta_0) \big) - a'\nabla_\theta g(X,\theta_0) \Big)
\xrightarrow[n\to\infty]{\P_{\theta_0}} 0.
\]
Lemma~\ref{lem_prob_1_109_Liese} will allow to prove uniform integrability of
\[
L_{\theta_0}(u_n) \Big( \|u_n\|^{-1} \big( g(X,\theta_0 + u_n) - g(X,\theta_0) \big) - a'\nabla_\theta g(X,\theta_0) \Big)\\
= Z_n Y_n.
\]
With
\begin{align*}
Z_n & = L_{\theta_0}(u_n)^{1/2}\\
Y_n & =
L_{\theta_0}(u_n)^{1/2} \Big( \|u_n\|^{-1} \big( g(X,\theta_0 + u_n) - g(X,\theta_0) \big) - a'\nabla_\theta g(X,\theta_0) \Big).
\end{align*}
From Lemma~\ref{lem_1_106_Liese}, $Z_n$ satisfies $\E_{\theta_0} \big( Z_n - 1 \big)^2 \longrightarrow 0$. For all $n$ there is a random variable
$\theta_n = \theta_0 + \lambda_n u_n$,
with $\lambda_n \in [0,1]$ such that
\[
\| u_n \|^{-1} \left( g(X,\theta_0+u_n) - g(X,\theta_0) \right) = \| u_n \|^{-1} u'_n \nabla_\theta g(X,\theta_n).
\]
Hence
\[
\| u_n \|^{-2} \big( g(X,\theta_0 + u_n) - g(X,\theta_0) \big)^2 \leqslant
\| \nabla_\theta g(X,\theta_n) \|^2.
\]
For $n$ large enough, $\theta_n\in U(\theta_0)$, and then
\begin{align*}
\E_{\theta_0} L_{\theta_0}(u_n)
\| u_n \|^{-2} \big( g(X,\theta_0 + u_n) - g(X,\theta_0) \big)^2
& \leqslant \E_{\theta_0 + u_n} \| \nabla_\theta g(X,\theta_n) \|^2\\
& \leqslant \sup_{(\theta,\theta') \in U(\theta_0)^2}\E_\theta \|\nabla_\theta g(X,\theta') \|^2.
\end{align*}
Moreover
\[
\E_{\theta_0} \big( L_{\theta_0}(u_n)^{1/2} a'\nabla_\theta g(X,\theta_0) \big)^2
\leqslant  \sup_{(\theta,\theta') \in U(\theta_0)^2} \E_\theta \|\nabla_\theta g(X,\theta') \|^2.
\]
Hence $\E_{\theta_0} Y_n^2 < \infty$. We deduce that $(Z_n Y_n, n \in \N)$ is uniformly integrable and hence
\[
\E_{\theta_0} L_{\theta_0}(u_n) \Big( \|u_n\|^{-1} \big( g(X,\theta_0 + u_n) - g(X,\theta_0) \big) - a'\nabla_\theta g(X,\theta_0) \Big) \underset{n\rightarrow \infty}{\longrightarrow} 0.
\]
We deduce
\[
\nabla_\theta \E_{\theta_0} g(X,\theta_0) = \E_{\theta_0} g(X,\theta_0) \dot{L}_{\theta_0} + \E_{\theta_0} \nabla_\theta g(X,\theta_0).
\]

The case $k > 1$ is deduced from the case $k=1$ by reasoning componentwise.
\end{proof}

\begin{hypo}\label{hypo_CR_pred_L2_diff}
Consider a model $(\mathcal{X},\mathcal{B},\P_\theta, \theta\in\Theta)$, $\theta_0\in\mathring{\Theta}$, a neighbourhood $U(\theta_0)$ of $\theta_0$ and a function $g : \mathcal{X}\times \Theta \to \R^k$, with $g(\cdot,\theta)$ measurable for all $\theta\in\Theta$, such that the following conditions hold.
\begin{enumerate}
\item\label{cond_famille_L2_diff} The family $(\P_\theta,\theta\in\Theta)$ is $\L^2$-differentiable at $\theta_0$, with derivative $\dot{L}_{\theta_0}$.
\item Fisher matrix information $I(\theta_0)$ is invertible.
\item\label{cond_sup_J_g} For all $\theta$, $\theta' \in U(\theta_0)$, $g(X,\cdot)$ is $\P_\theta$-almost surely differentiable at $\theta'$ and \[
\sup_{(\theta,\theta') \in U(\theta_0)^2} \E_\theta \| \J_\theta g(X,\theta') \|_{\M_{k,d}}^2 < \infty.
\]
\item\label{cond_sup_L2} $\sup_{(\theta,\theta') \in U(\theta_0)^2} \E_\theta L_{\theta, \theta'}^2 < \infty$
\end{enumerate}
Moreover consider a predictor $p(X)$ taking values in $\R^k$. There is $U(\theta_0)$, a neighbourhood of $\theta_0$, such that
\begin{enumerate}\setcounter{enumi}{4}
\item\label{cond_sup_p} $\sup_{\theta\in U(\theta_0)} \E_\theta \| p(X) \|_{\R^k}^2 < \infty$.
\end{enumerate}
\end{hypo}

We first state the inequality for unbiased predictors. Here we say that $p(X)$ a predictor of $g(X,\theta)$ is an \emph{unbiased predictor} if $\E_\theta (p(X)) = \E_\theta (g(X,\theta))$ for all $\theta \in \Theta$ (for other concepts of risk unbiasedness pertaining to prediction problems see \cite{nayak2010}). 

\begin{theoreme}\label{theo_CR_pred_unbiased}
Let $(\mathcal{X},\mathcal{B},\P_\theta, \theta\in\Theta)$ be a model, $\theta_0\in\mathring{\Theta}$, and $p(X)$ an unbiased predictor of $g(X,\theta)$ taking values in $\R^k$, that satisfies Assumption~\ref{hypo_CR_pred_L2_diff}.

Then the QEP of $p(X)$ at $\theta_0$ satisfies the following inequality.
\begin{equation}\label{equa_theo_ineg_CR_pred_L2_diff}
\E_{\theta_0} (p(X) - g(X,\theta_0))^{\times 2} \geqslant
G(\theta_0) I(\theta_0)^{-1} G(\theta_0)',
\end{equation}
with $G(\theta) =  \E_\theta\J_\theta g(X,\theta)$.
The equality holds in (\ref{equa_theo_ineg_CR_pred_L2_diff}) iff
\begin{equation*}%\label{equation_egalite_CR}
p(X) = g(X,\theta_0) + G(\theta_0) I(\theta_0)^{-1} \dot{L}'_{\theta_0},
\quad \P_{\theta_0}\text{-a.s.}
\end{equation*}
\end{theoreme}

\begin{proof}
The result follows from Lemma~\ref{lem_ineg_CR_L2_diff} with
\[
G(\theta_0) = \J_\theta \psi(\theta_0) - \E_{\theta_0} g(X,\theta_0) \dot{L}'_{\theta_0} = \J_\theta \E_{\theta_0} g(X,\theta_0) - \E_{\theta_0} g(X,\theta_0) \dot{L}'_{\theta_0},
\]
where $\J_\theta \psi(\theta_0) = \J_\theta \E_{\theta_0} g(X,\theta_0)$ because $p(X)$ is assumed unbiased.
Yet Proposition~\ref{prop_deriv_g_L2_diff} gives
\[
\J_\theta \E_{\theta_0} g(X,\theta_0) = \E_{\theta_0} \J_\theta g(X,\theta_0) + \E_{\theta_0} g(X,\theta_0) \dot{L}'_{\theta_0}.
\]
We deduce $G(\theta_0)  =  \E_{\theta_0}\J_\theta g(X,\theta_0)$.
\end{proof}

\begin{remarque}
The following assumptions are used by \cite{miyata2001} to prove (\ref{equa_theo_ineg_CR_pred_L2_diff}).
\begin{itemize}
\item The family $\left( P_\theta^{(X,Y)}, \, \theta\in\Theta \right)$ (of distributions of the couple $(X,Y)$) is $\L^2$-differentiable.
\item Fisher information matrix is invertible.
\item $\E_{\theta_0} g(X,\theta)^2$ is bounded for all $\theta$ in a neighbourhood of all fixed $\theta_0 \in \Theta$.
\item The predictor $p(X)$ is unbiased, $\E_\theta p(X)^2 < \infty$, and $\E_\theta Y^2 < \infty$.
\end{itemize}
It is interesting to remark that these assumptions refer to the variable $Y$, while in our approach the variable $Y$ only comes up through the conditional expectation $r(X,\theta) = \E_\theta[g(X,Y,\theta)|X]$ and then it is not refered to anymore.
\end{remarque}

\begin{theoreme}
Let $(\mathcal{X},\mathcal{B},\P_\theta, \theta\in\Theta)$ be a model and $\theta_0\in\mathring{\Theta}$. Let $r:\mathcal{X}\times\Theta\to\R^k$ such that, $\P_{\theta_0}$-almost surely, $\theta \mapsto r(X,\theta)$ is differentiable at $\theta_0$, and for all $\theta$, the function $x \mapsto r(x,\theta)$ is measurable. Let $p(X)$ be a predictor of $r(X,\theta)$ with bias $b(\theta)$.

Suppose $(\mathcal{X},\mathcal{B},\P_\theta, \theta\in\Theta)$, $\theta_0$, $p(X)$, and $g(X,\theta) = r(X,\theta) + b(\theta)$, satisfy Assumption~\ref{hypo_CR_pred_L2_diff}.

Then $\theta \mapsto b(\theta)$ is differentiable at $\theta_0$, $\theta \mapsto r(X,\theta)$ is $\P_{\theta_0}$-almost surely differentiable at $\theta_0$, and the QEP of $p(X)$ at $\theta_0$ satisfies the following inequality.
\begin{equation}\label{equa_theo_ineg_CR_biais_pred_L2_diff}
\E_{\theta_0} (p(X) - r(X,\theta_0))^{\times 2} \geqslant
b(\theta_0)^{\times 2} + G(\theta_0) I(\theta_0)^{-1} G(\theta_0)',
\end{equation}
with $G(\theta_0) = \E_{\theta_0} \J_\theta r(X,\theta_0) + \J_\theta b(\theta_0)$.
The equality holds in (\ref{equa_theo_ineg_CR_biais_pred_L2_diff}) iff
\begin{equation*}%\label{equation_egalite_CR}
p(X) = b(\theta_0) + r(X,\theta_0) + G(\theta_0) I(\theta_0)^{-1} \dot{L}'_{\theta_0},
\quad \P_{\theta_0}\text{-a.s.}
\end{equation*}
\end{theoreme}

\section{Efficient prediction}\label{section_pred_eff}

A predictor $p(X)$ is said \emph{efficient} when its QEP attains the Cramér-Rao bound.

\begin{theoreme}\label{theo_famille_expo_L2_diff}
Suppose $k=d$. Let $\Theta$ be a connected open set of $\R^d$.
Let $(\mathcal{X},\mathcal{B},\P_\theta, \theta\in\Theta)$ be a model, $g : \mathcal{X} \times \Theta \to \R^k$ and $p(X)$ an unbiased predictor of $g(X,\theta)$, that satisfy Assumption~\ref{hypo_CR_pred_L2_diff} for all $\theta\in\Theta$.

Suppose the following conditions hold.
\begin{enumerate}
\item $p(X)$ is efficient.
\item For all $\theta\in\Theta$, $G(\theta) = \E_\theta \J_\theta g(X,\theta)$ is invertible.
\item\label{cond_existence_gradient} There is $A : \Theta \to \R^k$ a differentiable function over $\Theta$, such that $(\J_\theta A(\theta))' = I(\theta)G(\theta)^{-1}$, for all $\theta\in\Theta$.
\item $\mathcal{X}$ is a topological space and $(\mathcal{X}, \mathcal{B})$ is a $\sigma$-compact space.
\item\label{cond_sup_sur_compact} For all compact sets $C\subset\mathcal{X}$, $\tilde{C} \subset \Theta$, $\sup_{x\in C, \theta \in \tilde{C}} \| \J_\theta g(x,\theta) \| < \infty$.
\item\label{cond_I_G_continuous} $\theta \mapsto I(\theta)$ and $\theta \mapsto G(\theta)$ are continuous.
\end{enumerate}

Then, for $\theta_0 \in \Theta$ fixed, there is a function $B:\mathcal{X}\times\Theta \to \R$, differentiable at $\theta\in\Theta$, such that for all $\theta\in\Theta$, for $\P_{\theta_0}$-almost all $x\in\mathcal{X}$,
\[
\frac{d\P_\theta}{d\P_{\theta_0}}(x) = \exp\big(A(\theta)'p(x) - B(x,\theta)\big),
\]
and $\nabla_\theta B(x,\theta) = (\J_\theta A(\theta))' g(x,\theta)$.
\end{theoreme}

%La démonstration is adaptée of celle of \cite{muller-funk1989} where le résultat is prouvé for les isimateurs.

\begin{proof}
%Let $\Theta_0$ une composante connexe of $\Theta$ and $\theta_0 \in\Theta_0$.
Let $\theta \in \Theta$.
The predictor $p(X)$ is efficient hence $\P_\theta$-a.s.
\begin{align*}
p(X) & = g(X,\theta) + \left( \E_\theta \J_\theta g(X,\theta) \right) I(\theta)^{-1}\dot{L}_\theta\\
& = g(X,\theta) + G(\theta) I(\theta)^{-1}\dot{L}_\theta.
\end{align*}
Hence
\[
\dot{L}_\theta = I(\theta) G(\theta)^{-1} \big( p(X) - g(X,\theta) \big).
\]
Let  $s\mapsto \theta_s$ be a continuously differentiable path from $\theta_0$ to $\theta$ with $s\in[0,1]$. This path exists because $\Theta$ is open and connected.
We set
\[
f(x) = \exp \Big( \int_0^1 \dot{\theta}'_s \dot{L}_{\theta_s}(x) \, ds \Big)
= \exp \Big( \int_0^1 \big( \dot{\theta}'_s I(\theta_s) G(\theta_s)^{-1} p(x) - \phi(s,x) \big) ds \Big),
\]
with
\begin{align*}
%\mathcal{A}(\theta)^T & = I(\theta)\left( \E_\theta \J_\theta g(X,\theta) \right)^{-1}\\
\phi(s,x) & = \dot{\theta}'_s I(\theta_s) G(\theta_s)^{-1} g(x,\theta_s).
\end{align*}
We prove that for all event $B\in\mathcal{B}$, the following equality holds
\[
\int_B f(X) \, d\P_{\theta_0} = \P_\theta(B).
\]
Since $\mathcal{B}$ is $\sigma$-compact, one may assume that $B$ is a compact set.
For $\P_\theta$-almost all $x\in\mathcal{X}$, $s \mapsto g(x,\theta_s)$ is differentiable over $[0,1]$ (we remove from $B$ the points $x$ for which differentiability does not hold). We set
\[
M = \sup_{x\in B, s \in [0,1]} \| \partial_s g(x,\theta_s) \|
\leqslant \sup_{s\in[0,1]} \| \dot{\theta_s} \| \sup_{x\in B, t \in \{ \theta_s, s \in [0,1] \}} \| \J_\theta g(x,t) \|.
\]
The first supremum of the right hand side is finite because $(\theta_s, s \in [0,1])$ is continuously differentiable. The second one is finite from condition~\ref{cond_sup_sur_compact}. Hence $M < \infty$.
Let $\varepsilon>0$ and $(R_i)_{i\in\mathbb{N}}$ be a partition of $\mathbb{R}^k$ in rectangles of diameters at most $\varepsilon$, and let 
\[
n = \left\lceil \frac{M}{\varepsilon} \right\rceil.
\]
For all $u\in\mathbb{N}^{n+1}$ we let
\[
S_u = \left\{ x\in\mathcal{X} \, | \, \forall i\in\{0,\ldots,n\}, \, g(  x, \theta_{i/n} ) \in R_{u_i} \right\}.
\]
We then define
\[
B_{i,u} = B\cap p^{-1}(R_i)\cap S_u.
\]
Let $x\in B_{i,u}$ and $s\in[0,1]$ then,
\begin{align*}
\| g(x,\theta_s) \|
& \leqslant \| g(x,\theta_{\lfloor s n \rfloor / n}) \|
+ \| g(x, \theta_{\lfloor s n \rfloor / n}) - g(x, \theta_s) \| \\
& \leqslant \sup_{y \in R_{u_{\lfloor sn \rfloor}}} \|y\|
+ M \left| \lfloor s n \rfloor / n - s \right| \\
& \leqslant \sup_{y \in R_{u_{\lfloor sn \rfloor}}} \|y\|
+ \frac{M}{n} \\
& \leqslant \sigma_u + \frac{M}{n} < \infty,
\end{align*}
with
\[
%\sigma_u = \sup_{\substack{i \in \{0,\ldots,n\} \\ y \in R_{u_i}}} \| y \|.
\sigma_u = \sup_{ 0 \leqslant i \leqslant n, \;  y \in R_{u_i}} \| y \|.
\]
We prove by contradiction that $\P_{\theta_0}(B_{i,u}) > 0$ iff $\P_\theta(B_{i,u}) > 0$. Without loss of generality, suppose that $\P_{\theta_s}(B_{i,u}) > 0$ for $s\in[0,1)$ and $\P_\theta(B_{i,u}) = 0$. We set $H(s) = \log \P_{\theta_s}(B_{i,u})$. From Proposition~\ref{prop_1_111_Liese}, $s \mapsto \P_{\theta_s}(B_{i,u})$ is differentiable over $[0,1]$, hence it is continuous over $[0,1]$. Hence
\[
\lim_{s \to 1^-} \P_{\theta_s}(B_{i,u}) = 0.
\]
And therefore
\begin{equation}\label{equa_lim_H}
\lim_{s \to 1^-} H(s) = -\infty.
\end{equation}
Besides $H$ is differentiable over $[0,1)$. Its derivative is
\[
h(s) = \frac{ \dot{\theta}'_s \nabla_\theta \P_{\theta_s}(B_{i,u}) }{ \P_{\theta_s} (B_{i,u}) }
= \frac{1}{\P_{\theta_s}(B_{i,u})} \dot{\theta}'_s \int_{B_{i,u}} \dot{L}_{\theta_s} d\P_{\theta_s}
= m(s|B_{i,u}) - \phi(s|B_{i,u}),
\]
where
\begin{align*}
m(s|B_{i,u}) & = \P_{\theta_s}(B_{i,u})^{-1} \int_{B_{i,u}} \dot{\theta}'_s I(\theta_s) G(\theta_s)^{-1} p(X) \, d\P_{\theta_s}, \\
\phi(s|B_{i,u}) & = \P_{\theta_s}(B_{i,u})^{-1} \int_{B_{i,u}} \phi(s,X) \, dP_{\theta_s}.
\end{align*}
We prove that $h(s)$ is bounded. The function $s \mapsto \dot{\theta}'_s I(\theta_s) G(\theta_s)^{-1}$ is continuous over $[0,1]$, from condition~\ref{cond_I_G_continuous}, hence
\[
c = \sup_{s\in[0,1]} \| \dot{\theta}'_s I(\theta_s) G(\theta_s)^{-1} \| < \infty.
\]
Let $x\in B_{i,u}$, then $p(x) \in R_i \cup \{ 0 \}$ hence
\[
| \dot{\theta}'_s I(\theta_s) G(\theta_s)^{-1} p(x) | \leqslant c \sup_{y\in R_i} \| y \| = c \rho_i.
\]
Hence $| m(s|B_{i,u}) | \leqslant c \rho_i$. From what precedes we deduce
\[
\left| \phi(s,x) \right| \leqslant c \, \| g(x,\theta_s) \|
\leqslant c (\sigma_u + M/n).
\]
Hence $\phi(s|B_{i,u}) \leqslant c (\sigma_u + M/n)$. We deduce that $h$ is bounded over $[0,1)$, which contradicts (\ref{equa_lim_H}). Hence $\P_{\theta_0}(B_{i,u}) > 0$ iff $\P_\theta(B_{i,u}) > 0$, which implies that the distributions $\P_\theta$ and  $\P_{\theta_0}$ are absolutely continuous with respect to each other.

One may write
\begin{align*}
\int_{B_{i,u}} f(X) \, d\P_{\theta_0} & = 
\int_{B_{i,u}} \exp \bigg( \int_0^1 \Big( \dot{\theta}'_s I(\theta_s) G(\theta_s)^{-1} p(X) - m(s|B_{i,u}) + h(s) \\
& \quad \quad + \phi(s|B_{i,u}) - \phi(s,X) \Big) ds \bigg) \, d\P_{\theta_0} \\
& = \int_{B_{i,u}} \exp \bigg( \int_0^1 \Big( \dot{\theta}'_s I(\theta_s) G(\theta_s)^{-1} p(X) - m(s|B_{i,u}) \Big) \, ds \\
& \quad \quad + \int_0^1 \big( \phi(s|B_{i,u}) - \phi(s,X) \big) \, ds \bigg) \, d\P_{\theta_0}\, \frac{ \P_\theta(B_{i,u}) }{ \P_{\theta_0}(B_{i,u}) }.
\end{align*}
For all $x\in B_{i,u}$, $\dot{\theta}'_s I(\theta_s) G(\theta_s)^{-1} p(x)$ lies in the image of $R_i$ by the application
\[
y \mapsto \dot{\theta}'_s I(\theta_s) G(\theta_s)^{-1} y.
\]
Same thing for $m(s|B_{i,u})$ which is the mean of $\dot{\theta}'_s I(\theta_s) G(\theta_s)^{-1} p(x)$ over $B_{i,u}$. Hence
\[
\left| \dot{\theta}'_s I(\theta_s) G(\theta_s)^{-1} p(X) - m(s|B_{i,u}) \right|
\leqslant \sup_{s\in[0,1]} \| \dot{\theta}'_s I(\theta_s) G(\theta_s)^{-1} \| \, \mathrm{diam}(R_i) \leqslant c \varepsilon.
\]
Hence for all $x\in B_{i,u}$,
\[
\Big| \int_0^1 \big( \dot{\theta}'_s I(\theta_s) G(\theta_s)^{-1} p(x) - m(s|B_{i,u}) \big) \, ds \Big| \mathbbm{1}_{B_{i,u}} \leqslant c \varepsilon.
\]
Moreover
\[
\phi(s|B_{i,u}) - \phi(s,x) =
\frac{ \dot{\theta}'_s I(\theta_s) G(\theta_s)^{-1} }{ \P_{\theta_s}(B_{i,u}) } \int_{B_{i,u}} \big( g(X,\theta_s) - g(x,\theta_s) \big) \, d\P_{\theta_s}.
\]
For $x$, $x'\in B_{i,u}$,
\begin{align*}
\| g(x,\theta_s) - g(x',\theta_s) \|
& \leqslant \| g(x,\theta_{\lfloor sn\rfloor/n}) - g(x',\theta_{\lfloor sn\rfloor/n}) \| \\
& \quad + \| g(x,\theta_{\lfloor sn\rfloor/n}) - g(x,\theta_s) \| \\
& \quad + \| g(x',\theta_{\lfloor sn\rfloor/n}) - g(x',\theta_s) \| \\ 
& \leqslant \mathrm{diam}(R_{u_{\lfloor sn\rfloor}}) + \frac{2 M}{n} \leqslant 3 \varepsilon.
\end{align*}
Hence
\[
\left| \phi(s|B_{i,u}) - \phi(s,x) \right|
\leqslant \sup_{s\in[0,1]} \| \dot{\theta}'_s I(\theta_s) G(\theta_s)^{-1} \| \times 3 \varepsilon = 3 c \varepsilon.
\]
Hence
\[
e^{-4c\varepsilon} \P_\theta(B) \leqslant
\int_B f(X) \, d\P_{\theta_0} \leqslant
e^{4c\varepsilon} \P_\theta(B),
\]
for all $\varepsilon>0$. And therefore $\int_B f(X) \, d\P_{\theta_0} = \P_\theta(B)$.
Hence, for $\P_{\theta_0}$-almost all $x\in\mathcal{X}$,
\[
\frac{ d\P_\theta }{ d\P_{\theta_0} }(x) =
\exp \left( A(\theta)' p(x) - B(x,\theta) \right),
\]
with
\begin{align*}
A(\theta)' & = \int_0^1 \dot{\theta}'_s I(\theta_s) G(\theta_s)^{-1} \, ds,\\
B(x,\theta) & = \int_0^1 \dot{\theta}'_s I(\theta_s) G(\theta_s)^{-1} g(x,\theta_s) \, ds.
\end{align*}
From condition~\ref{cond_existence_gradient} and the gradient theorem, $A(\theta)$ does not depend on  $(\theta_s, s\in[0,1])$, the chosen path. Yet
\[
\frac{ d\P_\theta }{ d\P_{\theta_0} }(x) = \log f(x) = \int_0^1 \big( \dot{\theta}'_s I(\theta_s) G(\theta_s)^{-1} p(x) - \phi(s,x) \big) ds,
\]
does not depend on it either, hence $B(x,\theta)$ does not depend on it. Therefore
\[
\nabla_\theta B(x,\theta) = I(\theta) G(\theta)^{-1} g(x,\theta) = (\J_\theta A(\theta))' g(x,\theta).
\]
\end{proof}

\begin{remarque}
In Theorem~\ref{theo_famille_expo_L2_diff} we did not assumed \emph{continuous} $\L^2$-differen\-tiability as \cite{muller-funk1989} did for their analogous result in the case of estimation.
If we add a condition of continuous $\L^2$-differentiability in Theorem~\ref{theo_famille_expo_L2_diff}, this makes possible to save somme assumptions. More precisely, the result of Theorem~\ref{theo_famille_expo_L2_diff} also holds under the following conditions.
\begin{enumerate}
\item The family $(\P_\theta, \theta\in\Theta)$ is continuously $\L^2$-differentiable and $\Theta$ is a connected open set of $\R^d$.
\item The matrix $I(\theta)$ is invertible for all $\theta \in \Theta$.
\item $p(X)$ is an unbiased efficient predictor of $g(X,\theta)$.
\item For all $\theta$, $\E_\theta \| p(X) \|^2 < \infty$.
\item For all $\theta\in\theta$, $G(\theta) = \J_\theta \E_\theta  g(X,\theta) - \E_\theta g(X,\theta) \dot{L}'_\theta$ is invertible, or equivalently, $\E_\theta \left( p(X) - g(X,\theta) \right)^{\times 2}$ is invertible.
\item There exists $A : \Theta \to \R^k$ a differentiable function over $\Theta$, such that $(\J_\theta A(\theta))' = I(\theta)G(\theta)^{-1}$, for all $\theta\in\Theta$.
\item $\mathcal{X}$ is a topological space and $(\mathcal{X}, \mathcal{B})$ is a $\sigma$-compact space.
\item For all compact set $C\subset\mathcal{X}$, $\tilde{C} \subset \Theta$, $\sup_{x\in C, \theta \in \tilde{C}} \| \J_\theta g(x,\theta) \| < \infty$.
\end{enumerate}
Conditions to have $G(\theta) = \E_\theta \J_\theta g(X,\theta)$ are not fulfilled anymore, hence we only get the expression $G(\theta) = \J_\theta \E_\theta  g(X,\theta) - \E_\theta g(X,\theta) \dot{L}'_\theta$. In the list of conditions above one saves conditions~\ref{cond_sup_J_g}, \ref{cond_sup_L2} and \ref{cond_sup_p} of Assumption~\ref{hypo_CR_pred_L2_diff} and condition~\ref{cond_I_G_continuous} of Theorem~\ref{theo_famille_expo_L2_diff}.
\end{remarque}

\begin{remarque}
The essential idea in the proof of Theorem~\ref{theo_famille_expo_L2_diff} is to cut the set $B$ with the family of subsets with the following form
\[
B_{i,u} = B\cap p^{-1}(R_i)\cap S_u,
\]
while for the result in the case of estimation, Müller-Funk \emph{et al.} \cite{muller-funk1989} took the family of subsets with the form $B_i = B\cap p^{-1}(R_i)$.
\end{remarque}

\begin{remarque}
In the particular case where $g$ does not depend on $X$, $g(X,\theta) = g(\theta)$, Theorem~\ref{theo_famille_expo_L2_diff} gives the well-known result that the existence of an efficient unbiased estimator implies the family is exponential.
\end{remarque}

\appendix

\section{$\L^2$-differentiable families}

We remind some defintions and results about $\L^2$-differentiable families of distributions, we refer to \cite{liese2008} p. 58 and next.
For $\theta$, $\theta_0$ in $\Theta$, any random variable $L_{\theta_0,\theta}$ taking values in $[0,+\infty]$ is called likelihood ratio of $\P_\theta$ with respect to $\P_{\theta_0}$ if, for all $A\in\mathcal{A}$,
\[
\P_\theta(A) = \int_A L_{\theta_0, \theta} d\P_{\theta_0} + \P_\theta \left( A \cap \{ L_{\theta_0,\theta} = +\infty \} \right).
\]
$L_{\theta_0,\theta}$ is a probability density of $\P_\theta$ with respect to $\P_{\theta_0}$ if and only if $\P_\theta \ll \P_{\theta_0}$. If $\nu$ is a measure over $\mathcal{A}$ that dominates $\{ \P_\theta, \P_{\theta_0} \}$ with $\{ f_\theta, f_{\theta_0} \}$ the corresponding densities then
\[
L_{\theta_0, \theta} = \frac{f_\theta}{f_{\theta_0}} \mathbbm{1}_{\{f_{\theta_0} > 0\}} + \infty \mathbbm{1}_{\{f_{\theta_0}=0, f_\theta>0\}},
\quad \{\P_\theta, \P_{\theta_0} \}\text{-a.s.}
\]
For all $\theta\in\Theta$, for all $u\in\R^d$ such that $u+\theta\in\Theta$, we set
\[
L_\theta(u) = L_{\theta, \theta + u}.
\]

\begin{definition}
The family $(\P_\theta, \theta\in\Theta)$ is said \emph{$\L^2$-differentiable} at $\theta_0\in\mathring{\Theta}$, if there is $U(\theta_0)$ a neighbourhood of $\theta_0$, such that for all $\theta\in U(\theta_0)$, $\P_\theta \ll \P_{\theta_0}$, and if there is $\dot{L}_{\theta_0}\in\L_{\P_{\theta_0}}^2(\R^d)$, called the $\L^2$-derivative of the model at $\theta_0$, such that as $u\to 0$,
\[
\E_{\theta_0} \left( L_{\theta_0}^{1/2}(u) - 1 - \frac{1}{2} u' \dot{L}_{\theta_0} \right)^2 = o(\| u \|_{\R^d}).
\]
The matrix $I(\theta_0) = \E_{\theta_0} \dot{L}'_{\theta_0} \dot{L}_{\theta_0}$ is called the \emph{Fisher information matrix} of the model at $\theta_0$.
\end{definition}

\begin{remarque}
If $\P_\theta \ll \nu$ for all $\theta\in\Theta$, then the family $(\P_\theta, \theta\in\Theta)$ is $\L^2$-differentiable at $\theta_0\in\mathring{\Theta}$, if and only if there is $\dot{f}_{\theta_0}\in\L_\nu^2(\R^d)$ such that as $u\to 0$,
\[
\int \left( \sqrt{f_{\theta_0+u}} - \sqrt{f_{\theta_0}} - \frac{1}{2} u' \dot{f}_{\theta_0} \right)^2 d\nu = o(\| u \|_{\R^d}).
\]
With $f_{\theta_0}$ and $f_{\theta_0 + u}$ the densities of $\P_{\theta_0}$ and $\P_{\theta_0 + u}$ with respect to $\nu$. We then have
\[
\dot{L}_{\theta_0} = \frac{\dot{f}_{\theta_0}(X)}{\sqrt{f_{\theta_0}(X)}},
\quad \P_{\theta_0}\text{-a.s.}
\]
Some authors call this property Hellinger-differentiability.
\end{remarque}

The following result is a recasting of Propositions~1.110 and 1.111 of Liese and Miescke (2008) \cite{liese2008}.

\begin{proposition}\label{prop_1_111_Liese}
Let $(\P_\theta, \theta\in\Theta)$ be a $\L^2$-differentiable family at $\theta_0\in\mathring{\Theta}$ with $\dot{L}_{\theta_0}$ the $\L^2$-derivative and let $\delta$ a r.v. taking values in $\R^k$ such that there is a neighbourhood $U(\theta_0)$ of $\theta_0$ with
\[
\sup_{\theta\in U(\theta_0)} \E_\theta \| \delta \|_{\R^k}^2 < \infty.
\]
Then $\psi : \theta \mapsto \E_\theta \delta$ is differentiable at $\theta_0$, and the jacobian marix of $\psi$ is
\[
\J_\theta \psi(\theta_0) = \E_{\theta_0} \big( \delta \dot{L}'_{\theta_0} \big).
\]
In particular, $\theta\in\Theta$, $\E_\theta \dot{L}_\theta = 0$.
\end{proposition}

We give the definition of \emph{continuous} $\L^2$-differentiability.

\begin{definition}
Let $(\P_\theta, \theta \in \Theta)$ be an $\L^2$-differentiable family over $\Theta$, with $\dot{L}_\theta$ as $\L^2$-derivative. We say that $(\P_\theta, \theta \in \Theta)$ is a \emph{continuously} $\L^2$-differentiable family over $\Theta$ if for all $\theta_0 \in \Theta$,
\[
\lim_{\theta\to\theta_0} \| L_{\theta,\theta_0}^{1/2} \dot{L}_\theta - \dot{L}_{\theta_0} \|^2 = 0.
\]
\end{definition}

% La  $\L^2$-differentiability continuous implique la continuité, voire la differentiability, en la variable $\theta$, of certaines fonctions découlant du model. La proposition suivante is utile in la démonstration du théorème qui la suit. Elle sera aussi utile for démontrer la généralisation of ce théorème à la prédiction, le Theorem~\ref{theo_famille_expo_L2_diff}.
% 
% \begin{proposition}
% Let $(\P_\theta, \theta \in \Theta)$ une famille continuously $\L^2$-differentiable sur $\Theta$, with $\dot{L}_\theta$ comme $\L^2$-derivative. Then
% \begin{enumerate}
% \item The function $\theta \mapsto I(\theta)$ is continuous.
% \item For all $B\in\mathcal{B}$, the function $\theta \mapsto \P_\theta(B)$ is continuously differentiable, of gradient $\nabla_\theta \P_\theta(B) = \E_\theta \mathbbm{1}_B \dot{L}_\theta$.
% \item Si $\delta$ is une v.a. such that $\theta \mapsto \E_\theta \| \delta \|^2$ is continuous, then $\theta \mapsto \E_\theta \delta \dot{L}'_\theta = \J_\theta \E_\theta \delta$ is continuous sur $\Theta$.
% \end{enumerate}
% \end{proposition}

The two following lemmas are useful for proving Proposition~\ref{prop_deriv_g_L2_diff} which allows to obtain the simpler form of the Cramér-Rao inequality for predictors in Theorem~\ref{theo_CR_pred_unbiased}.
The following result is Lemma~1.106 of Liese and Miescke (2008) \cite{liese2008}.

\begin{lemme}\label{lem_1_106_Liese}
Let $(\P_\theta,\theta\in\Theta)$ be a family of probability measures, $\Theta\subset\R^d$. Let $\theta_0\in\mathring{\Theta}$ and $U(\theta_0)$ be a neighbourhood of $\theta_0$, suppose that for all $\theta\in U(\theta_0)$, $\P_\theta \ll \P_{\theta_0}$. Then the family $(\P_\theta,\theta\in\Theta)$ is $\L^2$-differentiable at $\theta_0$, iff the two following conditions are fulfilled.
\begin{align*}
L_{\theta_0}(u) - 1 & = u' \dot{L}_{\theta_0} + o_{\P_{\theta_0}}(\|u\|)\\
\E_{\theta_0} \left( L_{\theta_0}^{1/2}(u) - 1 \right)^2 & = \frac{1}{4} u' I(\theta_0) u + o(\|u\|^2).
\end{align*}
\end{lemme}

The following lemma is useful to prove Proposition~\ref{prop_deriv_g_L2_diff}.

\begin{lemme}\label{lem_prob_1_109_Liese}
Let $X$, $X_n$, $Y_n$, $n=1,2,\ldots$ random variables such that $\E X^2 < \infty$, $\E(X_n-X)^2 \to 0$, and $\sup_{n\in\N} \E Y_n^2 < \infty$, then the sequence $(X_n Y_n, n\in\N)$ is uniformly integrable.
\end{lemme}

\begin{proof}
The convergence $\E(X_n-X)^2 \to 0$ implies $\exists n_0 \in \N$, $\sup_{n \geqslant n_0} \E X_n^2 < \infty$. We deduce
\[
\sup_{n \geqslant n_0} \E |X_n Y_n|
\leqslant \left ( \sup_{n \geqslant n_0} \E X_n^2 \right)^{1/2} \left ( \sup_{n \geqslant n_0} \E Y_n^2 \right)^{1/2} < \infty.
\]
Let $A$ be an event and $\varepsilon > 0$,
\[
\E |X_n Y_n| \mathbbm{1}_A
\leqslant
\left( \E X_n^2 \mathbbm{1}_A \right)^{1/2} \left( \E Y_n^2 \right)^{1/2}
\leqslant
\left( \E X_n^2 \mathbbm{1}_A \right)^{1/2} C^{1/2},
\]
with $C = \sup_{n\in\N} \E Y_n^2$.
Yet from $\E(X_n - X)^2 \to 0$ and $\E X^2 < \infty$, we deduce that the sequence $(X_n, n\in\N)$ is uniformly integrable (Theorem~\ref{theo_vitali}). Hence there are $n_0\in\N$ and $\alpha > 0$, such that for all $n\geqslant n_0$,
\[
\P(A) < \alpha \, \Rightarrow \, \E X_n^2 \mathbbm{1}_A < \frac{\varepsilon^2}{C}.
\]
We deduce, for all event $A$ such that $\P(A) < \alpha$, for all $n \geqslant n_0$, $\E |X_n Y_n| \mathbbm{1}_A < \varepsilon$. The sequence $(X_n Y_n, n\in\N)$ is hence equicontinuous. We deduce that it is uniformly integrable.
\end{proof}

\section{Uniform integrability and convergence}

\begin{definition}
We say that a family $\mathcal{F}$ of real r.v. is uniformly integrable if
\[
\sup_{X\in\mathcal{F}} \E \left( |X|\mathbbm{1}_{|X|>a} \right) \xrightarrow[a \to \infty]{} 0.
\]
\end{definition}

\begin{definition}
We say that a sequence of real r.v. $(X_n, n\in\N)$ is uniformly integrable if there is $n_0\in \N$ such that the family $(X_n, n\geqslant n_0)$ is uniformly integrable.
\end{definition}

\begin{proposition}
The family $\mathcal{F}$ is uniformly integrable iff
\begin{enumerate}
\item The family $\mathcal{F}$ is bounded in $\L^1$, \emph{i.e.} $\sup_{X\in\mathcal{F}} \E |X| < \infty$,
\item The family $\mathcal{F}$ is equicontinuous, \emph{i.e.} for all $\varepsilon>0$, there is $\alpha>0$, such that $\P(A) < \alpha$, implies $\sup_{X\in\mathcal{F}} \E \left( |X| \mathbbm{1}_A \right) < \varepsilon$.
\end{enumerate}
\end{proposition}

The result that follows is one of the versions of Vitali's theorem.

\begin{theoreme}\label{theo_vitali}
Let $p\in(0,+\infty)$, let $X$ be a r.v. and $(X_n,n\in\N)$ be a sequence of r.v. such that $\E X^p <\infty$ and for all $n$, $\E X_n^p <\infty$. Then the following conditions are equivalent.
\begin{enumerate}
\item $X_n \xrightarrow[n\to\infty]{\P} X$ and the sequence $(X_n^p,n\in\N)$ is uniformly integrable.
\item $\lim_{n\to\infty} \E (X_n-X)^p = 0$.
\end{enumerate}
\end{theoreme}

\bibliographystyle{imsart-nameyear}
\bibliography{../../Biblio}

\end{document}